\documentclass[reqno,12pt]{amsart}
\usepackage{indentfirst}
\usepackage{amsmath}
\usepackage[T1]{fontenc}
\usepackage{amssymb}			
\usepackage{amsthm}
\usepackage{amsfonts}
\usepackage{enumerate}
\usepackage{enumitem}
\usepackage{setspace}
\usepackage{mathtools}
\usepackage[body={7in,9.5in},top=0.5in, left=0.8in ]{geometry}
\newcommand{\ov}[1]{\overline{#1}}

\newcommand{\pref}[1]{\textup{(\ref{#1})}}

\newtheorem{thrm}{Theorem}

\newtheorem{thm}[thrm]{Theorem}
\newtheorem{definition}[thrm]{Definition}

\newtheorem{df}[thrm]{Definition}		

\newtheorem{cor}[thrm]{Corollary}
\newtheorem{rmk}[thrm]{Remark}

\newtheorem{exam}[thrm]{Example}

\usepackage[nobysame,initials]{amsrefs}

\newcommand{\R}{{\mathbb{R}}}
\newcommand{\C}{{\mathbb{C}}}

\newcommand{\N}{{\mathbb{N}}}

\title{on composition of muiltivariable formal power series}


\author{Motaz Mokatren}
\address[M. Mokatren]{Department of Electrical and Electronics Engineering\\
	Kinneret College\\
	Zemach, Emek HaYarden Mobile Post 15132\\
	Israel}

\email[M. Mokatren]{motaz\_mok@mx.kinneret.ac.il}

\begin{document}

\maketitle
\sloppy
\begin{abstract}
The paper  provides a necessary and sufficient condition for the composition of multivariable formal power series and present the Generalized Chain Rule for formal power series of multiple variables.
\end{abstract}
\section{Introduction}
Over the years, the study of formal power series composition, or functional composition, has remained a compelling area of research for mathematicians.\\
In 2002, Gan and Knox \cite{GanKnox} established a necessary and sufficient condition for the composition of formal power series of one variable. Building on this, in 2007, Gan introduced the Generalized Chain Rule for formal power series of one variable (see \cite{Gan1}). In 2022, Bugajewski, Galimberti, and Maćkowiak extended the theory by proving a necessary and sufficient condition for the composition of formal power series where the outer series is one variable and the inner series is multivariable, as presented in \cite{DAP}. Expanding further on this framework, in 2023, Gan, Bugajewski, and Maćkowiak introduced the Chain Rule for the composition of one variable formal power series with multivariable formal power series (see Lemma 5.2 of \cite{JCP}).\\
This paper presents new results on key aspects of the theory of multivariable formal power series: a necessary and sufficient condition for the existence of the composition of multivariable formal power series, and the Generalized Chain Rule for multivariable formal power series.\\
Composition plays a crucial role in solving ordinary and partial differential equations (see \cite{ode}). One notable application of the composition of formal power series is in the construction of Riordan matrices (see \cite{SH}).\\
The paper is organized as follows. The next section sets notation and gives basic definitions. Section 3 presents the proof of the necessary and sufficient condition for the existence of the composition of multivariable formal power series , and the Generalized Chain Rule for multivariable formal power series. Section 4 presents conclusions and further work.
\section{Definitions, conventions and some basic facts}
In this section we are going to collect some basic definitions and results which will be needed in
the sequel.\\
Symbol $\N$ denotes the set of all positive integers and $\N_0:=\N\cup\{0\}$. Symbol $ \mathbb{Z}$ denotes the set of all integers. Let $\mathbb{F}$ stands for the field $\R$ or $\C$. If $n\in \N$, then $[n]:=\{1,\ldots,n\}$. Let us fix $n\in{\mathbb{N}}$ and denote by
$\hat{\alpha}_k$ the set of all nonnegative
integer solutions $\alpha_1,\ldots,\alpha_n$ of $\alpha_1+\ldots+\alpha_n=k$ for $k\in{\mathbb{N}_0}$, that is
$$
\hat{\alpha}_k:=\{\alpha=(\alpha_1,\ldots, \alpha_n)\in{\mathbb{N}_{0}^n}:\, \alpha_1+\ldots+\alpha_n=k\}.
$$
For each $\alpha\in{\N^{n}_0}$ there is exactly one $k\in \N_0$ for which $\alpha\in{\hat{\alpha}_k}$, we denote such a $k$ by $k(\alpha)$.
\begin{rmk}
(\cite{S}, p.25) The number of elements of $\hat{\alpha}_k$ is $\binom{n+k-1}{n-1}$.
\end{rmk}
We say that $\alpha\in{\N_0^n}$ is lexicographically less (greater) than $\beta\in{\N_0^n}$, $\alpha\neq \beta$, if $\alpha_i<\beta_i$ ($\alpha_i>\beta_i$) for the  first coordinate $i\in [n]$ for which $\alpha_i\neq  \beta_i$.
\begin{df}
 (\cite{GanM2})
Let $\mathbb{F}$ be a field and $\mathbb{F}[X]=\mathbb{F}[x_1,...,x_n]$ be the polynomial ring of n variables, such that each monomial in the ring is  $X^{\alpha}=x^{\alpha_1}_1x^{\alpha_2}_2...x^{\alpha_n}_n$, where
$\alpha=(\alpha_1,....,\alpha_n) \in \mathbb{N}_{0}^n $. A formal power series, briefly (fps) of n variables is defined as
\begin{equation}
  f=\sum_{\alpha}a_\alpha X^{\alpha}=\sum_{\alpha_1,....,\alpha_n \geq 0} {a_{(\alpha_1,....,\alpha_n)}{x^{\alpha_1}_1x^{\alpha_2}_2...x^{\alpha_n}_n }},
 \end{equation}
  where $a_\alpha$ is a map from $\mathbb{N}_0^n$ to $\mathbb{F}$.
\end{df}
\begin{rmk}
For a fps of one variable, we simply identify $X=(x_1)$ with $x$, where $x$ is the name of the variable, so we write $x^n$ instead of $X^{n}$.
\end{rmk}
\begin{df}
The degree of each monomial $X^{\alpha}=x^{\alpha_1}_1x^{\alpha_2}_2...x^{\alpha_n}_n$ in the ring $\mathbb{F}[x_1,...,x_n]$ is defined as $|\alpha|=\sum_{i=1}^n{\alpha_i}$.
\end{df}
Moreover, the multiplication in $\mathbb{F}[X]$ is defined as follows.
\begin{df}
Let  $f=\sum_{\alpha}a_\alpha X^{\alpha}$, $g=\sum_{\beta}b_{\beta} X^{\beta}$ be two fps. Then
$$fg=\sum_{\alpha}(\sum_{\beta,\gamma:\; \gamma+\beta=\alpha }a_{\beta}b_{\gamma})X^{\alpha}.$$
\end{df}
Actually, it is the Cauchy product of two formal power series:
\begin{equation}
fg=\sum_{\alpha}(\sum_{\beta}a_{\beta}b_{\alpha-\beta})X^{\alpha}.
\end{equation}
\begin{thm}
(\cite{Hunk2019})
The set of all formal power series forms an integral domain with respect to the addition and the multiplication.
\end{thm}
Let $\theta:=(0,...,0)$ and $e_{i}=(0,..,1,..,0)$ be an $n$-tuple vector, with $1$ on the $i^{th}$ place and zero otherwise.
\begin{df}
A fps $f=\sum_{\alpha}a_{\alpha}X^{\alpha}$ is said to be a unit, if $a_{\theta}\neq{0}$, and a nonunit, if $a_{\theta}=0$.
\end{df}
\begin{df}
Let $f=\sum_{\alpha}a_\alpha X^{\alpha}$ be a fps of n variables and $k\in{\N_0}$, the $k$-th block of $f$ is a fps $f[k]$ of n variables whose coefficients are given by $f[k]_\alpha:=a_\alpha$, $\alpha\in{\hat{\alpha}_k},$ and $f[k]_\alpha:=0$, otherwise.
\end{df}
\begin{rmk}
By the Multinomial Theorem we get $(f[0]+\ldots +f[k])^n=\underbrace{(f[0]+ \ldots + f[k])\ldots (f[0]+ \ldots + f[k])}_{n\times}=\sum \frac{n!}{v_0! \ldots  v_k!}f[0]^{v_0}\ldots f[k]^{v_k}$, where $k:=k(\alpha)$ and the sum extends over nonnegative integer solutions of $v_0+v_1+\ldots+v_k=n$. Notice that the only terms of the sum that contribute to the $k$-th block of $f^n$, that is, $f^n[k]$ are those for which $v_1+2v_2+\ldots+kv_k=k$. Therefore, it follows for every $\alpha\in{\N^{n}_{0}}$ \begin{equation}\label{eqn:powers} f^n_\alpha=
\sum \frac{n!}{v_0! \ldots  v_k!}(f[0]^{v_0}\ldots f[k]^{v_k})_\alpha,\end{equation} where $k=k(\alpha)$ and the sum extends over nonnegative integer solutions of the system \begin{equation}\label{eqn:cndtns}\left\{\begin{array}{l}v_0+v_1+\ldots+v_k=n\\v_1+2v_2+\ldots+kv_k=k.\end{array}\right.\end{equation}
\end{rmk}
\begin{df}
(\cite{Hunk2019})
Let $ f=\sum_{\alpha}a_\alpha X^{\alpha}=\sum_{\alpha_1,....,\alpha_n \geq 0} {a_{(\alpha_1,....,\alpha_n)}{x^{\alpha_1}_1x^{\alpha_2}_2...x^{\alpha_n}_n }}$ be a power series of n variables. The formal partial
derivative of $f$ with respect to $x_j$ is defined as:
$$D_j(f)=\sum_{\alpha_1,....,\alpha_n \geq 0} {(\alpha_j+1)a_{(\alpha_1,...,\alpha_{j-1},\alpha_j+1,\alpha_{j+1}...,\alpha_n)}{x^{\alpha_1}_1x^{\alpha_2}_2...x^{\alpha_n}_n }}.$$
\end{df}
The m-th derivative is defined in the natural way and is denoted by $D_j^m(f)$.
\begin{thm} (\cite{Hunk2019})
It holds:
$$D_j(f+g)=D_j(f)+D_j(g),$$
$$D_j(fg)=D_j(f)g+fD_j(g),$$
$$D_j(f^m)=mf^{m-1}D_j(f),m\in{\mathbb{N}},$$
\end{thm}
\begin{thm}
  Let $f_1,f_2,...,f_n$ be fps of n variables. Then
  $$D_j(\prod_{i=1}^{n}f_i)=\sum_{i=1}^n{D_j(f_i)\prod_{l\neq{i}}f_l}.$$
\end{thm}
\begin{proof}
The proof is trivial by the mathematical induction in view of n.
\end{proof}
 We use the conventions that $\prod_{i\in{\emptyset}}:=1$ and $\sum_{i\in \emptyset}:=0$. Let us denote $I=(x_1,x_2,...,x_n)$ to be the identity of the composition.

\section{ The composition of two formal power series}
First, we recall the composition of a fps of one variable and with a fps of several variables.
\begin{definition}(\cite{DAP})\label{df:composition}
For $f=\sum_{\alpha}a_\alpha X^{\alpha}$ , and  $g(x)=\sum_{n=0}^{\infty}b_nx^n$, the composition of $g$ with $f$ is a fps of several variables $h=g\circ f$ defined by
$$h:=\sum_{\alpha}\underbrace{\left(\sum_{n=0}^\infty b_n a^{n}_\alpha\right)}_{h_\alpha:=}X^\alpha,$$
provided that the coefficients $h_\alpha$ exist, that is, if the series defining $h_\alpha$ converge for every $\alpha$.
\end{definition}
\begin{rmk}
Theorem 3.1 of \cite{DAP}, provides a necessary and sufficient condition for the existence of such a composition.
\end{rmk}
The following definitions are vital for our paper.
\begin{df}
For each n-tuple of nonnegative integers $\alpha=(\alpha_1,\alpha_2,...,\alpha_n)$ and $G=(g_1,g_2,...,g_n)$  - a vector of fps with n variables over $\mathbb{F}$, we define a monomial of fps as:
$$G^\alpha=g_1^{\alpha_1}g_2^{\alpha_2}\cdots{}g_n^{\alpha_n}.$$
\end{df}

\begin{df}
\cite{haz}
For $\alpha=(\alpha_1,\ldots, \alpha_n)\in{\mathbb{N}_{0}^n}$. Let $f=\sum_{\alpha}a_\alpha X^{\alpha}$  and  $G=(g_1,...,g_n)$ be a vector of fps of n variables over $\mathbb{F}$, every
component of which is $g_i=\sum_{\alpha}b_\alpha^iX^{\alpha}$. The formal composition of $f$ with $G$, if it exists, is a fps of several variables $h=f\circ G$ defined by
$$h=f\circ{G}=f\circ{(g_1,...,g_n)}=\sum_{\alpha}a_\alpha G^{\alpha}=\sum_{\alpha_1,....,\alpha_n \geq 0} {a_{(\alpha_1,....,\alpha_n)}{g^{\alpha_1}_1g^{\alpha_2}_2...g^{\alpha_n}_n }}.$$
We distinguish the trivial case where all $g_i = 0$ and define the composition as
$$h=f\circ{\theta}=1.$$
\end{df}
\begin{df}
\cite{haz}
Let $F=(f_1,...,f_m)$ and $G=(g_1,...,g_n)$ be two vectors of fps of n variables over $\mathbb{F}$, where $f_i=\sum_{\alpha}a^{i}_\alpha X^{\alpha}$, and
$g_i=\sum_{\alpha}b_\alpha^iX^{\alpha}$. The formal composition of $G$ with $F$, if it exists, is a vector of fps of several variables $H=F\circ G$ defined by
$$H=F\circ{G}=(f_1\circ{G},...,f_n\circ{G}).$$
\end{df}
\begin{rmk}
The composition of multivariable formal power series cannot be composed component by component. However, partial composition is possible for multivariable functions. For more details, see \cite{ber}.\\
Example:\\
Let $f=\frac{1}{1-x_1x_2}$,
and $g_1=1-x_1$,$g_2=1-x_2$.
Both $f\circ(g_1,x_2)$ and $f\circ(x_1,g_2)$ exist but $f\circ(g_1,g_2)$ doesn't.
\end{rmk}
The following is the main result of this section, which generalizes both the one variable case and the case where the outer series is one variable and the inner series is multivariable.
\begin{thm}
For $\alpha=(\alpha_1,\ldots, \alpha_n)\in{\mathbb{N}_{0}^n}$. Let $f=\sum_{\alpha}a_\alpha X^{\alpha}$  and  $G=(g_1,...,g_n)$ be a vector of fps of n variables over $\mathbb{F}$, every
component of which is $g_i=\sum_{\alpha}b_\alpha^iX^{\alpha}$ and is not a constant fps. The formal composition of $f$ with $G$ exists if and only if
\begin{equation}\label{comp}
\lim_{|\beta|\rightarrow{}\infty}\sum_{|\alpha|\geq{}|\beta|}{\binom{\alpha_1}{\beta_1}\binom{\alpha_2}{\beta_2}\cdots{}\binom{\alpha_n}{\beta_n}a_{\alpha}(b_{\theta}^{1})^{\alpha_1-\beta_1}(b_{\theta}^{2})^{\alpha_2-\beta_2}\cdots{}(b_{\theta}^{n})^{\alpha_n-\beta_n}}\in{\mathbb{F}},
\end{equation}
for $\beta=(\beta_1,...\beta_n)\in{\N_0^n}$. If $|\alpha|\geq |\beta|$, it follows that $\alpha_i\geq \beta_i$ for every $i\in [n]$.
\end{thm}
\begin{proof}
For simplicity, the proof is presented for the two variable case, which, in a sense, is analogous to the n variable case.\\
Let $G=(g_1,g_2)$ be a vector of fps of two variables. Assume that $f\circ{G}$ is well defined. This implies that
$$\lim_{\beta1,\beta_2\to{\infty}}\sum_{\alpha_1=0}^{\beta_1}\sum_{\alpha_2=0}^{\beta_2}{a_{(\alpha_1,\alpha_2)}g_1^{\alpha_1}g_{2}^{\alpha_2}}\in{\mathbb{F}}.$$
Thus, for every $\gamma\in{\N^{2}_{0}}$ and for every $k_1,k_2\in \N_0$, the following limit holds:
\begin{equation}\label{eqn:lim}
\lim_{\beta_1,\beta_2\to \infty}\sum_{\alpha_1=0}^{\beta_1}\sum_{\alpha_2=0}^{\beta_2}a_{(\alpha_1,\alpha_2)}(w[k_1]^{\alpha_1})_{\gamma}(w[k_2]^{\alpha_2})_{\gamma}\in{\mathbb{F}},\end{equation}
where $w[k_1]^{\alpha_1}$ and $w[k_2]^{\alpha_2}$ are defined as
$$w[k_1]^{\alpha_1}:=\sum \frac{\alpha_1!}{v_0! \ldots  v_{k_1}!}g_1[0]^{v_0}\ldots g_1[k_1]^{v_{k_1}}$$
and
$$w[k_2]^{\alpha_2}:=\sum \frac{\alpha_2!}{v_0! \ldots  v_{k_2}!}g_2[0]^{v_0}\ldots g_2[k_2]^{v_{k_2}},$$
where the sums are taken over the sets of nonnegative integer solutions $v=(v_0,v_1,\ldots,v_{k_1})$ and $v=(v_0,v_1,\ldots,v_{k_2})$, respectively, to the system (4).
 Notice that for each $k_1,k_2\in \N_0$, the existence of the limit \pref{eqn:lim} is equivalent to the existence of the limit
 \begin{equation}\label{eqn:lim1}\lim_{\beta_1,\beta_2\to \infty}\sum_{\alpha_1=k_1}^{\beta_1}\sum_{\alpha_2=k_2}^{\beta_2}a_{(\alpha_1,\alpha_2)}(w[k_1]^{\alpha_1})_{\gamma}(w[k_2]^{\alpha_2})_{\gamma}\in{\mathbb{F}}.
\end{equation}
Let us denote $w_{k_1,k_2}^{\beta1,\beta_2}(x_1,x_2):=\sum_{\alpha_1=k_1}^{\beta_1}\sum_{\alpha_2=k_2}^{\beta_2}a_{(\alpha_1,\alpha_2)}(w[k_1]^{\alpha_1})_{\gamma}(w[k_2]^{\alpha_2})_{\gamma}(x_1,x_2)$.
If $k_1=k_2=0$, then the given limit is equivalent to the existence of the limit $\lim_{\beta_1,\beta_2\to \infty}\sum_{\alpha_1=0}^{\beta_1}\sum_{\alpha_2=0}^{\beta_2}a_{(\alpha_1,\alpha_2)}(b_{\theta}^{1})^{\alpha_1}(b_{\theta}^{2})^{\alpha_2}$.\\
Let $k\in \N_0$ and $f$ be a fps of n variables. Define
$$R^n_k:=\left\{(v_0,\ldots,v_k)\in \N^{k+1}_0:\, \sum_{i=0}^kv_i=n\text{ and }\sum_{i=1}^kiv_i=k\right\}=\{v^{n,1},\ldots,v^{n,r^n_k}\},
$$
where $v^{n,j}=(v^{n,j}_0,v^{n,j}_1,\ldots,v^{n,j}_k)\in \N^{k+1}_0$ and $r^n_k$ denotes the number of elements of $R^n_k$,
$$m_{k,j}^n(f):=\frac{1}{v^{n,j}_1! \ldots v^{n,j}_k!}f[1]^{v^{n,j}_1}\ldots f[k]^{v^{n,j}_k},\quad j\in [r^n_k],$$
$$W^n_{k,s}:=\{j\in [r^n_k]:\, \sum_{i=1}^{k}v^{n,j}_i=s\},\quad s\in [k],$$
$$d^{n}_{k,s}(f):=\sum_{j\in W^n_{k,s}}m^n_{k,j}(f).$$
For $\beta_1\geq k_1$ and $\beta_2\geq k_2$, the following holds:

\begin{multline*}
    w_{k_1,k_2}^{\beta_1,\beta_2} =
    \sum_{\alpha_1=k_1}^{\beta_1} \sum_{\alpha_2=k_2}^{\beta_2}
    a_{(\alpha_1,\alpha_2)}
    \left(
        \sum_{v\in R^{\alpha_1}_{k_1}}
        \frac{\alpha_1!}{v_0! \ldots v_{k_1}!} g_1[0]^{v_0} \ldots g_1[k_1]^{v_{k_1}}
        \sum_{v\in R^{\alpha_2}_{k_2}}
        \frac{\alpha_2!}{v_0! \ldots v_{k_2}!} g_2[0]^{v_0} \ldots g_2[k_2]^{v_{k_2}}
    \right)
\end{multline*}

\begin{multline*}
    = \sum_{\alpha_1=k_1}^{\beta_1} \sum_{\alpha_2=k_2}^{\beta_2}
    a_{(\alpha_1,\alpha_2)}
    \left(
        \sum_{j=1}^{r^{\alpha_1}_{k_1}} (b^{1}_{\theta})^{v^{\alpha_1,j}_0}
        \frac{\alpha_1!}{v^{\alpha_1,j}_0!}
        \frac{1}{v^{\alpha_1,j}_1! \ldots v^{\alpha_1,j}_{k_1}!}
        g_1[1]^{v^{\alpha_1,j}_1} \ldots g_1[k_1]^{v^{\alpha_1,j}_{k_1}}
\right) \cdot{} \\
    \left(
        \sum_{j=1}^{r^{\alpha_2}_{k_2}} (b^2_{\theta})^{v^{\alpha_2,j}_0}
        \frac{\alpha_2!}{v^{\alpha_2,j}_0!}
        \frac{1}{v^{\alpha_2,j}_1! \ldots v^{\alpha_2,j}_{k_2}!}
        g_2[1]^{v^{\alpha_2,j}_1} \ldots g_2[k_2]^{v^{\alpha_2,j}_{k_2}}
    \right)=\\
    =\sum_{\alpha_1=k_1}^{\beta_1}\sum_{\alpha_2}^{\beta_2}a_{(\alpha_1,\alpha_2)}\left(\sum_{j=1}^{r^{\alpha_1}_k}(b^1_{\theta})^{v^{\alpha_1,j}_0}\frac{\alpha_1!}{v^{\alpha_1,j}_0!}m^{\alpha_1}_{k_1,j}(g_1)\right)\cdot{}\left(\sum_{j=1}^{r^{\alpha_2}_{k_2}}(b^2_{\theta})^{v^{\alpha_2,j}_0}\frac{\alpha_2!}{v^{\alpha_2,j}_0!}m^{\alpha_2}_{k_2,j}(g_2)\right)=\\
    \sum_{\alpha_1=k_1}^{\beta_1}\sum_{\alpha_2=k_2}^{\beta_2}a_{(\alpha_1,\alpha_2)}\left(\sum_{s_1=1}^{k_1}\left(\sum_{j\in W^{\alpha_1}_{k_1,s_1}}(b^1_{\theta})^{v^{\alpha_1,j}_0}\frac{\alpha_1!}{v^{\alpha_1,j}_0!}m^{\alpha_1}_{k_1,j}(g_1)\right)\right)\cdot{}\left(\sum_{s_2=1}^{k_2}\left(\sum_{j\in W^{\alpha_2}_{k_2,s_2}}(b^2_{\theta})^{v^{\alpha_2,j}_0}\frac{\alpha_2!}{v^{\alpha_2,j}_0!}m^{\alpha_2}_{k_2,j}(g_2)\right)\right)=(\star).
\end{multline*}
Observe that for $\alpha_i\geq k_i$ with $i\in[2]$, the equalities  $W^{\alpha_i}_{k_i,s}=W^{\bar{\alpha}_i}_{k_i,s}$, $m_{k_i,j}^{\alpha_i}(g_i)=m_{k_i,j}^{\bar{\alpha}_i}(g_i)$, $d^{\alpha_i}_{k_i,s}(g_i)=d^{\bar{\alpha}_i}_{k_i,s}(g_i)$ hold whenever $\bar{\alpha_i}\geq \alpha_i$. Consequently, when $\bar{\alpha_i}\geq{\alpha_i}\geq k_i$, we adopt the simplified notation $W_{k_i,s}$, $m_{k_i,j}(g_i)$, $d_{k_i,s}(g_i)$ for every $i\in [2]$ instead.\\
 From the equalities $v^{\alpha_1,j}_0=\alpha_1-v^{\alpha_1,j}_1-\ldots-v^{\alpha_1,j}_{k_1}$ and $v^{\alpha_2,j}_0=\alpha_2-v^{\alpha_2,j}_1-\ldots-v^{\alpha_2,j}_{k_2}$, we obtain
  \begin{multline}\label{eq:lim2}
  (\star)=\sum_{\alpha_1=k_1}^{\beta_1}\sum_{\alpha_2}^{\beta_2}a_{(\alpha_1,\alpha_2)}\left(\sum_{s_1=1}^{k_1}\left(\sum_{j\in W_{k_1,s_1}}(b^1_{\theta})^{v^{\alpha_1,j}_0}\frac{\alpha_1!}{v^{\alpha_1,j}_0!}m_{k_1,j}(g_1)\right)\right)\cdot{}\left(\sum_{s_2=1}^{k_2}\left(\sum_{j\in W_{k_2,s_2}}(b^2_{\theta})^{v^{\alpha_2,j}_0}\frac{\alpha_2!}{v^{\alpha_2,j}_0!}m_{k_2,j}(g_2)\right)\right)=\\
  =\sum_{\alpha_1=k_1}^{\beta_1}\sum_{\alpha_2=k_2}^{\beta_2}a_{(\alpha_1,\alpha_2)}\left(\sum_{s_1=1}^{k_1}\frac{\alpha_1!}{(\alpha_1-s_1)!}(b^1_{\theta})^{\alpha_1-s_1} d_{k_1,s_1}(g_1)\right)\cdot{}\left(\sum_{s_2=1}^{k_2}\frac{\alpha_2!}{(\alpha_2-s_2)!}(b^2_{\theta})^{\alpha_2-s_2} d_{k_2,s_2}(g_2)\right)=\\
  =\sum_{s_1=1}^{k_1}\sum_{s_2=1}^{k_2} s_1!s_2!\left(\sum_{\alpha_1=k_1}^{\beta_1}\sum_{\alpha_2=k_2}^{\beta_2}{\alpha_1 \choose s_1}{\alpha_2 \choose s_2} a_{(\alpha_1,\alpha_2)}(b^1_{\theta})^{\alpha_1-s_1}(b^2_{\theta})^{\alpha_2-s_2}\right )d_{k_1,s_1}(g_1)d_{k_2,s_2}(g_2).
  \end{multline}
  Hence, condition \eqref{comp} is sufficient for the existence of the composition.\\
Next, we prove the necessity part. Assume that the composition exists, so $\lim_{\beta,1,\beta_2 \to \infty}(w_{k_1,k_2}^{\beta1,\beta_2})_{\alpha}<\infty$, for every $\alpha\in {\N^2_0}$. We will carry out the proof using double induction with respect to $k_1$ and $k_2$.\\
Clearly, the base case of the induction holds because, when $k_1=k_2=0$ then the constant term of the composition is well defined if and only if the series converges.\\
Suppose that the equivalence holds for some $k_1,k_2\in \N_0$, which ensures the convergence of the series
$\sum_{\alpha_1=s_1}^{\beta_1}\sum_{\alpha_2=s_2}^{\beta_2}{\alpha_1 \choose s_1}{\alpha_2 \choose s_2} a_{(\alpha_1,\alpha_2)}(b^1_{\theta})^{\alpha_1-s_1}(b^2_{\theta})^{\alpha_2-s_2}$, where $s_i\in{\{0\}\cup{[k_i]}}$ for every $i\in [2]$.

 Let $l_i$ for each $i\in [2]$ denote the least positive integer for which there exists $\alpha\in \hat{\alpha_{l_i}}$ with $g_{i}[l_i]_{\alpha}\neq 0$. Let us fix $m_1,m_2\in \N$. Since $g_1$,$g_2$ are non constant fps it has been proven in \cite{DAP} that for every $i\in [2]$:
$$d_{m_il_i,m_i}(g_i)(x_1,x_2)=\frac{1}{m_i!}g_i[l_i]^{m_i}(x_1,x_2).$$
The induction step is carried out separately for each index. Specifically, for $m_1=k_1+1$ and $\beta_1\geq (k_1+1)l_1$, $m_2=k_2+1$ and $\beta_2\geq (k_2+1)l_2$. Without loss of generality, we will prove the statement for the first index. By \pref{eq:lim2}, we obtain
 \begin{multline*} w^{\beta_1,\beta_2}_{(k_1+1)l_1,k_2}(x_1,x_2)=\sum_{s_1=1}^{(k_1+1)l_1}\sum_{s_2=1}^{k_2} s_1!s_2!\cdot{}\\\left(\sum_{\alpha_1=(k_1+1)l_1}^{\beta_1}\sum_{\alpha_2=k_2}^{\beta_2}{\alpha_1 \choose s_1}{\alpha_2 \choose s_2} a_{(\alpha_1,\alpha_2)}(b^1_{\theta})^{\alpha_1-s_1}(b^2_{\theta})^{\alpha_2-s_2}\right )d_{(k_1+1)l_1,s_1}(g_1)d_{k_2,s_2}(g_2)=\\
 =\sum_{s_1=1}^{k_1+1}\sum_{s_2=1}^{k_2} s_1!s_2!\left(\sum_{\alpha_1=(k_1+1)l_1}^{\beta_1}\sum_{\alpha_2=k_2}^{\beta_2}{\alpha_1 \choose s_1}{\alpha_2 \choose s_2} a_{(\alpha_1,\alpha_2)}(b^1_{\theta})^{\alpha_1-s_1}(b^2_{\theta})^{\alpha_2-s_2}\right )d_{(k_1+1)l_1,s_1}(g_1)d_{k_2,s_2}(g_2)=\\
		 \left[\sum_{s_1=1}^{k_1}\sum_{s_2=1}^{k_2} s_1!s_2!\left(\sum_{\alpha_1=(k_1+1)l_1}^{\beta_1}\sum_{\alpha_2=k_2}^{\beta_2}{\alpha_1 \choose s_1}{\alpha_2 \choose s_2} a_{(\alpha_1,\alpha_2)}(b^1_{\theta})^{\alpha_1-s_1}(b^2_{\theta})^{\alpha_2-s_2}\right )d_{(k_1+1)l_1,s_1}(g_1)d_{k_2,s_2}(g_2)\right]+
\\
\sum_{s_2=1}^{k_2} (k_1+1)!s_2!\left(\sum_{\alpha_1=(k_1+1)l_1}^{\beta_1}\sum_{\alpha_2=k_2}^{\beta_2}{\alpha_1 \choose k_1+1}{\alpha_2 \choose s_2} a_{(\alpha_1,\alpha_2)}(b^1_{\theta})^{\alpha_1-(k_1+1)}(b^2_{\theta})^{\alpha_2-s_2}\right )d_{(k_1+1)l_1,k_1+1}(g_1)d_{k_2,s_2}(g_2)=
\\
	 \left[\sum_{s_1=1}^{k_1}\sum_{s_2=1}^{k_2} s_1!s_2!\left(\sum_{\alpha_1=(k_1+1)l_1}^{\beta_1}\sum_{\alpha_2=k_2}^{\beta_2}{\alpha_1 \choose s_1}{\alpha_2 \choose s_2} a_{(\alpha_1,\alpha_2)}(b^1_{\theta})^{\alpha_1-s_1}(b^2_{\theta})^{\alpha_2-s_2}\right )d_{(k_1+1)l_1,s_1}(g_1)d_{k_2,s_2}(g_2)\right]+	\\
\sum_{s_2=1}^{k_2} s_2!\left(\sum_{\alpha_1=(k_1+1)l_1}^{\beta_1}\sum_{\alpha_2=k_2}^{\beta_2}{\alpha_1 \choose k_1+1}{\alpha_2 \choose s_2} a_{(\alpha_1,\alpha_2)}(b^1_{\theta})^{\alpha_1-(k_1+1)}(b^2_{\theta})^{\alpha_2-s_2}\right )g_1[l_1]^{k_1+1}d_{k_2,s_2}(g_2).
		\end{multline*}
Using the fact that if $\ov{\alpha}\in \hat{\alpha_{l_1}}$ is the unique minimal element of $g_1[l_1]_{\alpha}\neq0$ with respect to the lexicographic ordering. Then  $(g_1[l_1]^{k_1+1})_{(k_1+1)\ov{\alpha}}=g_{1_{\ov{\alpha}}}^{k_1+1}$. Since $g_{1_{\ov{\alpha}}}\neq 0$, and applying the induction hypothesis, we conclude that the limit $\lim_{\beta_1,\beta_2\to \infty} (w^{\beta_1,\beta_2}_{(k_1+1)l_1,k_2})_{(k_1+1)\ov{\alpha}}$ exists by assumption, and all sums
$\sum_{\alpha_1=s_1}^{\beta_1}\sum_{\alpha_2=s_2}^{\beta_2}{\alpha_1 \choose s_1}{\alpha_2 \choose s_2} a_{(\alpha_1,\alpha_2)}(b^1_{\theta})^{\alpha_1-s_1}(b^2_{\theta})^{\alpha_2-s_2}$ converge.
\end{proof}
After establishing the necessary and sufficient conditions for the existence of the composition, it is able to formulate the Chain Rule. However, before proceeding, these definition is needed.

\begin{df}
(\cite{Hunk2019}, Def 4.3) A multi-index sequence $(f_1,f_2,...)$ of lexicographic order fps such that $f_i=\sum_{\alpha}a^i_{\alpha}X^\alpha$, admits addition if for each n-tuple $(\beta_1,...,\beta_n)$ there
exists a positive integer N such that
$$a^{i}_{\alpha}=0,$$
for all $i\geq{N}$ and $0\leq{\alpha_1}\leq{\beta_1},...,0\leq{\alpha_n}\leq{\beta_n}$.
\end{df}

\begin{thm}(Chain Rule for the composition of multivariable fps with a vector of multivariable fps).\\
 Let $f=\sum_{\alpha}a_\alpha X^{\alpha}$  and  $G=(g_1,...,g_n)$ be a vector of fps of n variables over $\mathbb{F}$ ,every component is $g_i=\sum_{\gamma}b_\gamma^iX^{\gamma}$. Then
\begin{equation}
  D_j(f\circ{G})=\sum_{i=1}^{n}{(D_i(f)\circ{G})D_j(g_i)},
  \end{equation}
 provided that the sequence of the series $(G^\alpha:=\prod_{i=1}^n{g_{i}^{\alpha_i}})_{\alpha}$, $\alpha=(\alpha_1,...,\alpha_n)$, admits addition.
\end{thm}
\begin{proof}
Denote for every $\gamma_i\in{}\N$,  $g_i^{\gamma_i}=\sum_{\gamma}b_{\gamma}^{i^{(\gamma_i)}}X^{\gamma}$. Then\\
$G^\alpha=\prod_{i=1}^n{g_{i}^{\alpha_i}}$ (Def.14).\\
By Lemma 2.9 of \cite{DAP}, we obtain
$$G^\alpha=\prod_{i=1}^n{\sum_{\gamma}[{\sum_{\beta_1,...,\beta_{\alpha_i}\in{\N_0^n}\;: \beta_1+\cdots+\beta_{\alpha_i}=\gamma}b_{\beta_1}^i\cdots{}b_{\beta_{\alpha_i}}^{i}]X^{\gamma}}}.$$

Let us denote $G^\alpha=\sum_{\gamma}{c^{\alpha}_{\gamma}X^\gamma}$.\\
By the definition of the formal composition f with G, we get
$$f\circ{G}=\sum_{\alpha}a_\alpha G^{\alpha}=\sum_{\alpha}a_\alpha \sum_{\gamma}{c^{\alpha}_{\gamma}X^\gamma}.$$
By Theorem 4.4 of \cite{Hunk2019}, because that the sequence of series $(G^{\alpha})_{\alpha\in{\N_0^{n}}}$ admits addition, then the derivative of the series is the sum of the derivatives. Therefore
$$D_{j}(f\circ{G})=\sum_{\alpha}a_\alpha \sum_{\gamma}{c^{\alpha}_{\gamma+e^j}(\gamma_j+1)X^\gamma}=\sum_{\alpha}a_{\alpha}D_{j}(G^\alpha).$$

According to the definition of the formal partial derivative, we obtain

$$D_{i}(f)=\sum_{\alpha}{(\alpha_i+1)a_{\alpha+e_i}X^\alpha}.$$
However,
$$D_{i}(f)\circ{G}=\sum_{\alpha}{(\alpha_i+1)a_{\alpha+e_i}G^\alpha},$$
and
$$D_j(g_{i})=\sum_{\alpha}{(\alpha_j+1)b_{\alpha+e_j}^iX^\alpha}.$$
If we examine the left-hand side of (7) and applying Th.11, we get
$$\sum_{\alpha}a_{\alpha}D_{j}(G^\alpha)=\sum_{\alpha}a_{\alpha}[\sum_{i=1}^n{D_j(g_i^{\alpha_i})\prod_{k\neq{i}}g^{\alpha_k}_k}]=\sum_{i=1}^n[\sum_{\alpha}a_{\alpha}{D_j(g_i^{\alpha_i})\prod_{k\neq{i}}g^{\alpha_k}_k}].$$
Applying the third property of Th.10, we obtain
$$\sum_{i=1}^n[\sum_{\alpha}a_{\alpha}{D_j(g_i^{\alpha_i})\prod_{k\neq{i}}g^{\alpha_k}_k}]=\sum_{i=1}^n[\sum_{\alpha=(\alpha_1,...,\alpha_n)}\alpha_ia_{\alpha}{g^{\alpha_i-1}D_j(g_i)\prod_{k\neq{i}}g^{\alpha_k}_k}].$$
Substituting $\alpha=\alpha+e_i$, we get
$$\sum_{i=1}^n[\sum_{\alpha}\alpha_ia_{\alpha}{g^{\alpha_i-1}D_j(g_i)\prod_{k\neq{i}}g^{\alpha_k}_k}]=\sum_{i=1}^n[\sum_{\alpha}(\alpha_i+1)a_{\alpha+e_i}{g^{\alpha_i}D_j(g_i)\prod_{k\neq{i}}g^{\alpha_k}_k}].$$
However, we have
$$\sum_{i=1}^n[\sum_{\alpha}(\alpha_i+1)a_{\alpha+e_i}{g^{\alpha_i}D_j(g_i)\prod_{k\neq{i}}g^{\alpha_k}_k}]=\sum_{i=1}^n[\sum_{\alpha}(\alpha_i+1)a_{\alpha+e_i}G^{\alpha}D_j(g_i)],$$
which leads to the equality
$$\sum_{i=1}^n[\sum_{\alpha}(\alpha_i+1)a_{\alpha+e_i}G^{\alpha}D_j(g_i)]=\sum_{i=1}^{n}{(D_i(f)\circ{G})D_j(g_i)}.$$
\end{proof}
\begin{df}
(\cite{Mat97})
The formal Jacobian matrix of $G=(g_1,...,g_n)$ where $g_i$ are fps, with respect to $x_1, . . . ,x_n$ is defined to be the $n\times{n}$ matrix whose $(i, j )$-entry is $\frac{\partial{g_i}}{dx_j}$, denoted by
\begin{equation}
J_G(x_1,...,x_n):=\frac{\partial{G}}{\partial X}=\frac{\partial{(g_1,...,g_n)}}{\partial(x_1,...x_n)}=[D_{j}(g_{i})]_{1\leq{i},j\leq{n}}:=[\frac{\partial{g_i}}{\partial{x_j}}]_{1\leq{i},j\leq{n}}.
\end{equation}
\end{df}
\begin{rmk}
Let f be a fps of n variables and G be a vector of n fps of n variables, such that $f\circ{G}$ exists. By the Chain Rule, we get:
$$J_{f\circ{G}}=J_{f}(G)\cdot{}J_{G}.$$
\end{rmk}

\section{Further work}
Composition of formal power series is fundamental to the structure and definition of Riordan matrices, serving as the primary operation that governs their behavior. This operation not only underpins the matrix entries themselves, but also dictates how Riordan matrices interact under multiplication and inversion.\\
The following theorem establishes the existence and sufficient conditions for the composition inverse. While this result is widely recognized (see \cite{haz} and \cite{arn}). The proof presented here offers a more concise and simple demonstration.
\begin{thm}
Let $G=(g_1,...,g_n)$ be a vector of nonunit fps of n variables over $\mathbb{F}$, $g_i=\sum_{\alpha}a^i_{\alpha}X^{\alpha}$.
Then $G$ has a unique composition inverse $G^{-1}=(g^{-1}_1,...,g^{-1}_n)$ where $g^{-1}_i=\sum_{\alpha}b^i_{\alpha}X^{\alpha}$ for every $1\leq{}i\leq{n}$, if and only if $det(J_{G}(\theta))\neq{0}.$
\end{thm}
\begin{proof}
Let us examine the following formula:
$$G^{-1}\circ{G}=(g^{-1}_1\circ{G},g^{-1}_2\circ{G},\cdots{},g^{-1}_n\circ{G})=I=(x_1,x_2,...,x_n).$$
The above formula holds for every $1\leq{}i\leq{n}$, that is $g^{-1}_i\circ{G}=x_i$, so
\begin{equation}
g^{-1}_i\circ{G}=\sum_{\alpha}b^i_{\alpha}G^{\alpha}=b^i_{\theta}+b^i_{e_1}g_1+b^i_{e_2}g_2+...+b^i_{e_n}g_n+\sum_{\alpha:\;|\alpha|\geq{2}}b^i_{\alpha}G^{\alpha}=x_i.
\end{equation}
By comparing the constant term, we get that $b^i_{\theta}=0$ for every $i\in{\{1,...,n\}}$.\\
If we compare the coefficients which indices belonging to $\hat{\alpha}_1$, we get a linear system of n equations as follows:
\begin{equation}\label{eqn:cndtns}\left\{\begin{array}{l}a^1_{e_1}b^{1}_{e_1}+a^1_{e_2}b^{2}_{e_1}+\ldots+a^1_{e_n}b^{n}_{e_1}=1\\a^2_{e_1}b^{1}_{e_2}+a^2_{e_2}b^{2}_{e_2}+\ldots+a^2_{e_n}b^{n}_{e_2}=1\\\vdots
\\a^n_{e_1}b^{1}_{e_n}+a^n_{e_2}b^{2}_{e_n}+\ldots+a^n_{e_n}b^{n}_{e_n}=1\end{array}\right.,\end{equation}\\
which is equivalent to the following system:
\[
\begin{bmatrix}
a^1_{e_1}  & a^1_{e_2} & \cdots & a^1_{e_n}\\
a^2_{e_1} & a^2_{e_2} & \cdots & a^2_{e_n}\\
\vdots & \vdots & \ddots & \vdots\\
a^n_{e_1} & a^n_{e_2} & \cdots & a^n_{e_n}
\end{bmatrix}
\cdot
\begin{bmatrix}
b^1_{e_1}  & b^1_{e_2} & \cdots & b^1_{e_n}\\
b^2_{e_1} & b^2_{e_2} & \cdots & b^2_{e_n}\\
\vdots & \vdots & \ddots & \vdots\\
b^n_{e_1} & b^n_{e_2} & \cdots & b^n_{e_n}
\end{bmatrix}
=
\begin{bmatrix}
1 & 0 & \cdots & 0\\
0 & 1 & \cdots & 0\\
\vdots & \vdots & \ddots & \vdots\\
0 & 0 & \cdots & 1
\end{bmatrix}.
\]

Notice that the two matrices on the left hand side are $J_G(\theta)\cdot{}J_{G^{-1}}(\theta)$. Therefore, we obtain  $det(J_G(\theta)\cdot{}J_{G^{-1}}(\theta))=det(J_G(\theta))\cdot{}det(J_{G^{-1}}(\theta))=1$.\\
Hence, $J_G(\theta)\neq{0}$ is a necessary condition for the existence of the inverse composition.\\
In general, if we compare the coefficients which indices belonging to $\hat{\alpha}_j$, where $j\geq{2}$, we get for every $1\leq{}i\leq{n}$, homogeneous systems of  linear equations in which the unknowns are
$b^{i}_{\alpha}$. So, it possesses at least one solution that is called a trivial solution where the value of each coefficient is $0$. Then, indeed the composition inverse exists.\\
Using the fact that the composition is associative (see Prop 4.1.4 of \cite{TB}), then we can consider the set of all vectors of fps of n variables, each having a composition inverse. It forms a group under composition by
the
definition. Then by (\cite{Jhon} Prop 3.1.2 (b)), the inverse is unique.\\
\end{proof}
In (\cite{SH}, p.232), it was stated that the above theorem holds if and only if $g_i(\theta)=0$ for all $1\leq{}i\leq{n}$ and $det(J_{G}(\theta))\neq{0}$. However, this statement is false. Let us consider a
counterexample as follows:\\
\begin{exam}
Given $G(x_1,x_2)=(1+x_1+x_2,1+x_1-x_2)$. So, ${g_1}(x_1,x_2)=1+x_1+x_2$ and ${g_2}(x_1,x_2)=1+x_1-x_2$.\\
  Notice that $G^{-1}(x_1,x_2)=(-1+x_1/2+x_2/2 , x_1/2-x_2/2)$  is the composition inverse of $G$. That is:
    $$g^{-1}_1\circ{G}=-1+(1+x_1+x_2)/2+(1+x_1-x_2)/2=x_1,$$
    and
    $$g^{-1}_2\circ{G}=(1+x_1+x_2)/2-(1+x_1-x_2)/2=x_2.$$
\end{exam}
\begin{cor}
The existence of a composition inverse does not necessarily apply exclusively to nonunit fps vectors.
\end{cor}
For future work, a natural question is how to extend the notion of Riordan matrices to include all unit fps, even in the one variable case. While it is known that the set of Riordan matrices forms a group under multiplication, extending this structure to arbitrary unit fps remains mysterious, as the associative property fails in general for such an extension.

\end{document}